\documentclass[12pt]{article}%
\usepackage{amsmath}
\usepackage{amsfonts}%
\usepackage{amssymb}%
\usepackage{graphicx}
\providecommand{\U}[1]{\protect \rule{.1in}{.1in}}
\newtheorem{theorem}{Theorem}

\newtheorem{definition}[theorem]{Definition}

\newtheorem{lemma}[theorem]{Lemma}

\newtheorem{proposition}[theorem]{Proposition}
\newtheorem{remark}[theorem]{Remark}

\newenvironment{proof}[1][Proof]{\noindent \textbf{#1.} }{\  \hfill $\square$ }
\newdimen \dummy
\dummy=\oddsidemargin
\begin{document}

\title{Dirichlet Problem for Coupled Elliptic Systems}
\author{Azeddine Baalal\thanks{a.baalal@fsac.ac.ma} \ and \ Mohamed
Berghout\thanks{moh.berghout@gmail.com} \\Department of Mathematics-Laboratory MACS \\Faculty of Sciences A\"{\i}n Chock, University Hassan II of Casablaca\\B.P 5366 Maarif Casablanca 20100 - Morocco}
\date{2016}
\maketitle

\begin{abstract}
Let $\Omega$ be a bounded domain in $%
%TCIMACRO{\U{211d} }%
%BeginExpansion
\mathbb{R}
%EndExpansion
^{d}$ $\left(  d\geq2\right)  $ pretty regular. In this paper we solve
the\ variational Dirichlet problem for a class of quasi-linear elliptic systems.

\end{abstract}

{\LARGE Introduction}

In the classical theory of the Laplace equation several main parts of
mathematics are joined in a fruitful way: Calculus of Variations, Partial
Differential Equations, Potential Theory, Function Theory (Analytic
Functions), Mathematical Physics and Calculus of Probability. This is the
strength of the classical theory. The $p$-Laplace equation occupies a similar
position, when it comes to non-linear phenomena.

Problem involving the $p-$Laplacian operator appears in pure mathematics such
as the theory of quasiregular and quasiconformal mapping as well as in applied
mathematics. Indeed, it intervenes in numerous fields in experimental
sciences: nonlinear reaction-diffusion problems, dynamics of populations,
non-Newtonian fluids, ect....for this reasons the word \textquotedblleft%
$p$-Laplacian\textquotedblright \ has become a key word in nonlinear analysis
and problems involving this second order quasilinear operator are now
extensively studied in the literature. This elliptic operator it generalizes
the usual Laplace operator $\triangle=\triangle_{2} $ whose study has been
widely discussed in recent decades but the lack of the Hilbert structure of
the space $W_{0}^{1,p}\left(  \Omega \right)  $ when passing from $p=2$ to
$p\neq2$ make his study very difficult.

The notion of the solution of quasilinear homogeneous Dirichlet problem
associated with the $p$-poisson equation
\[
\left \{
\begin{array}
[c]{c}%
-\triangle_{p}u=f\text{ \ in }\Omega \text{ ;}\\
u=0\text{ on }\partial \Omega \text{.}%
\end{array}
\right.
\]
is always understood in a weak sense, precisely is a function $u\in$
$W_{0}^{1,p}\left(  \Omega \right)  $.

In literature, there exists numerous papers dedicated to the study of
equations and systems involving the $p$-laplacian operator. In fact the study
of scalar equations had really started in the middle of 80s by M. \^{O}tani
\cite{MO} in one dimension then in dimension $n$ by F. de Th\'{e}lin
\cite{F.D} who obtained the first results on the equation of the form:
$\ -\triangle_{p}u=\lambda u^{\gamma-1}$, after the existence and the
uniqueness of radial solutions in $%
%TCIMACRO{\U{211d} }%
%BeginExpansion
\mathbb{R}
%EndExpansion
^{n}$ have showed independently by the last author and W. M. Ni. Serrin
\cite{MNS}. this result has been generalized by M. \^{O}tani \cite{MO01} to
any arbitrary open subset of $%
%TCIMACRO{\U{211d} }%
%BeginExpansion
\mathbb{R}
%EndExpansion
^{n}$. In 1987, F. de Th\'{e}lin \cite{FD1} has extended these results to the
equation of the type $\triangle_{p}u=g(x,u)$ where $g$ is a function
controlled by polynomial functions in $u$. In addition, there are other
results on the uniqueness were stated by J. I. D\`{\i}az and J. E. Saa
\cite{JES} in 1987 for the equation $-\triangle_{p}u=f(x,u)$.

The case of systems presents a new challenge and leads to tremendous
complications related to the coupling and lot of work done in this area for
instance we cite \cite{FD2}, \cite{RFF}, \cite{TJV}, \cite{J.C} ,
\cite{D.G.F}, \cite{FMST} \ and \cite{A.C.T}.

Let $\Omega$ be a bounded domain in $%
%TCIMACRO{\U{211d} }%
%BeginExpansion
\mathbb{R}
%EndExpansion
^{d}$ $\left(  d\geq2\right)  $ pretty regular, and let $\mathcal{L}_{1},$
$\mathcal{L}_{2}$ are a quasi-linear elliptic differential operators in
divergence form
\[
\left \{
\begin{array}
[c]{c}%
\mathcal{L}_{1}\left(  u,v\right)  :=-\triangle_{p}u+\varphi \left(
.,u,v\right)  \text{ \ }\\
\mathcal{L}_{2}\left(  u,v\right)  :=-\triangle_{p}v+\psi \left(  .,u,v\right)
\text{ }\
\end{array}
\right.
\]
where $\triangle_{p}u:=\operatorname{div}\left(  \left \vert \nabla
u\right \vert ^{p-2}\nabla u\right)  $ $\ $denotes the $p$-Laplacian operator
and $\varphi,\psi:\Omega \times%
%TCIMACRO{\U{211d} }%
%BeginExpansion
\mathbb{R}
%EndExpansion
\times%
%TCIMACRO{\U{211d} }%
%BeginExpansion
\mathbb{R}
%EndExpansion
\rightarrow%
%TCIMACRO{\U{211d} }%
%BeginExpansion
\mathbb{R}
%EndExpansion
$ are given Carath\'{e}odory functions satisfying:

\begin{itemize}
\item[$\left(  H_{1}\right)  $] $\left \vert \varphi \left(  x,u,v\right)
\right \vert \leq a_{1}\left \vert u\right \vert ^{p-1}+a_{2}\left \vert
v\right \vert ^{p-1}$ ;

$\left \vert \psi \left(  x,u,v\right)  \right \vert \leq b_{1}\left \vert
v\right \vert ^{p-1}+b_{2}\left \vert u\right \vert ^{p-1}$ .
\end{itemize}

where $a_{1},a_{2},b_{1},b_{2}$ \ are positive constants.

\begin{itemize}
\item[$\left(  H_{2}\right)  $] $s\mapsto \varphi \left(  x,s,t\right)  $ and
$s\mapsto \psi \left(  x,s,t\right)  $ are increasing for all $t\in%
%TCIMACRO{\U{211d} }%
%BeginExpansion
\mathbb{R}
%EndExpansion
$;

$t\mapsto \varphi \left(  x,s,t\right)  $ and $t\mapsto \psi \left(  x,s,t\right)
$ are increasing for all $s\in%
%TCIMACRO{\U{211d} }%
%BeginExpansion
\mathbb{R}
%EndExpansion
$.
\end{itemize}

Our aim in this paper is to solve the Dirichlet problem
\begin{equation}
\left \{
\begin{array}
[c]{c}%
\mathcal{L}_{1}\left(  u,v\right)  :=-\triangle_{p}u+\varphi \left(
.,u,v\right)  =0,\text{ \ in }\Omega \text{ };\\
\mathcal{L}_{2}\left(  u,v\right)  :=-\triangle_{p}v+\psi \left(  .,u,v\right)
\text{ }=0,\text{ in }\Omega \text{ };\\
u=h,v=k,\text{ on }\partial \Omega.
\end{array}
\right. \label{1eq1}%
\end{equation}
with a continous data boundary.

This paper consists of three sections. First, we give a some definitions for
the weak subsolutions, supersolution and solution of the system(\ref{1eq1})
and we prove that if a pair of functions $\left(  u,v\right)  $ is a
supersolution (resp. subsolution) of (\ref{1eq1}), then the pair of functions
$\left(  u+\alpha,v+\beta \right)  $(resp. $\left(  u-\alpha,v-\beta \right)  $)
is also a supersolution (resp. subsolution) of (\ref{1eq1}) for every
$\alpha>0$ and $\beta \geq0$. In section two we show that the comparison
principale is hold for $\mathcal{L}_{1}$ and $\mathcal{L}_{2}$. After this
preparation we are able in section three to solve the Dirichlet problem. So at
first by the Shaulder's fixed point theorem we prove the existence of
solutions to the associated variational problem for $\Omega$ with $\left \vert
\Omega \right \vert $ is small enough, after we solve the Dirichlet problem with
continuous data boundary for $\Omega$ $\ p- $regular with $\left \vert
\Omega \right \vert $ is small enough . Finally in the general case we will
approach $\Omega$ by an compact $K$ such that $\Omega \setminus K$ $\ $be
$p-$regular with $\left \vert \Omega \right \vert $ is small enough. Then by an
argument of collection of $K$ with a $p-$regular set who have a rather small
Lebesgue measure and by choiceing a sequence of functions and using a diagonal
extraction process; we show that the limit of a sequence of functions it's the solution.

\textbf{Notation}

Throughout of this paper we will use the following notation: $%
%TCIMACRO{\U{211d} }%
%BeginExpansion
\mathbb{R}
%EndExpansion
^{d}$ is the real Euclidean $d-$space, $d\geq2$. $\Omega$ is a open bounded
domaine of $%
%TCIMACRO{\U{211d} }%
%BeginExpansion
\mathbb{R}
%EndExpansion
^{d}$ pretty regular. For every $p\in \left[  1;+\infty \right[  $ we denote by
$p^{^{\prime}}=\frac{p}{p-1}$ the H\"{o}lder conjugate of $p$. For a mesurable
set $X$ and for $p\geq1$, $L^{p}\left(  X\right)  $ is the $p^{th}-$power
Lebesgue space defined on $X$ and $L^{p^{^{\prime}}}\left(  X\right)  $
denotes the dual space of $L^{p}\left(  X\right)  $. For an open set $U$ of $%
%TCIMACRO{\U{211d} }%
%BeginExpansion
\mathbb{R}
%EndExpansion
^{d}$, we denote by $\mathcal{C}^{k}\left(  U\right)  $ the set of functions
which $k-$th derivative is continuous for $k$ positive integer, $\mathcal{C}%
^{\infty}\left(  U\right)  =\cap_{k\geq1}\mathcal{C}^{k}\left(  U\right)  $
and by $\mathcal{C}_{c}^{\infty}\left(  U\right)  $ the set of all functions
in $\mathcal{C}^{\infty}\left(  U\right)  $ with compact support. The Sobolev
space $\mathcal{W}^{1,p}\left(  U\right)  $ is the Banach space of all
functions $u\in L^{p}\left(  U\right)  $ whose gradient in the distribution
sense $\nabla u\in \left(  L^{p}\left(  U\right)  \right)  ^{d}$, equipped with
the norm $\left \Vert u\right \Vert _{1,p}=\left \Vert u\right \Vert
_{p}+\left \Vert \nabla u\right \Vert _{p}$ , $\mathcal{W}_{loc}^{1,p}\left(
U\right)  $ is the space of all functions $u\in \mathcal{W}^{1,p}\left(
U\right)  $ for every open $U\subset \overline{U}\subset \Omega$ .
$\mathcal{W}_{0}^{1,p}\left(  U\right)  $ is the closure of $\mathcal{C}%
_{c}^{\infty}\left(  U\right)  $ in $\mathcal{W}^{1,p}\left(  U\right)  $
relatively to its norm (for more details see fo example \cite{MaZi97},
\cite{Mazj85}, \cite{Maz2011}, \cite{FD2012}, \cite{H.brez}). For $E\subset%
%TCIMACRO{\U{211d} }%
%BeginExpansion
\mathbb{R}
%EndExpansion
^{d}$ mesurable we denotes by $\left \vert E\right \vert $ the Lebesgue measure
of $E$. If $A\subset%
%TCIMACRO{\U{211d} }%
%BeginExpansion
\mathbb{R}
%EndExpansion
^{d}$ we denote by $\overline{A}$ the topological closure of $A$ and by
$\partial A$ the topological boundary of $A$. for a function $f$ we denote
$f^{+}:=\max \left \{  f,0\right \}  $, and we design by $\lim \sup$ (resp.
$\lim \inf$) the upper limit (resp. lower limit) of a real function.
$\mathcal{C}\left(  X\right)  $ the set of all continuous functions on $X$.
The order on the set of pairs of functions on a set $M$ is the usual order
product:
\[
\left(  f,g\right)  \leq \left(  h,k\right)  \Leftrightarrow f\leq h\text{ and
}g\leq k\text{.}%
\]

\section{Solution of (\ref{1eq1})}

\begin{definition}
\bigskip We say that a pair of functions $\left(  u,v\right)  \in
\mathcal{W}_{loc}^{1,p}\left(  \Omega \right)  \times \mathcal{W}_{loc}%
^{1,p}\left(  \Omega \right)  $ is a weak solution of (\ref{1eq1}) if
\[
\left \{
\begin{array}
[c]{c}%
\int_{\Omega}\left \vert \nabla u\right \vert ^{p-2}\nabla u\nabla \eta
_{1}dx+\int_{\Omega}\varphi \left(  x,u,v\right)  \eta_{1}dx=0,\\
\int_{\Omega}\left \vert \nabla v\right \vert ^{p-2}\nabla v\nabla \eta
_{2}dx+\int_{\Omega}\psi \left(  x,u,v\right)  \eta_{2}dx=0,
\end{array}
\right.
\]
for all $\eta_{1}$, $\eta_{2}\in \mathcal{W}_{0}^{1,p}\left(  \Omega \right)  $.

We say that a pair of functions $\left(  u,v\right)  \in \mathcal{W}%
_{loc}^{1,p}\left(  \Omega \right)  \times \mathcal{W}_{loc}^{1,p}\left(
\Omega \right)  $ is a supersolution (resp. subsolution) of (\ref{1eq1}) if
\[
\left \{
\begin{array}
[c]{c}%
\int_{\Omega}\left \vert \nabla u\right \vert ^{p-2}\nabla u\nabla \eta
_{1}dx+\int_{\Omega}\varphi \left(  x,u,v\right)  \eta_{1}dx\geq0\\
\int_{\Omega}\left \vert \nabla v\right \vert ^{p-2}\nabla v\nabla \eta
_{2}dx+\int_{\Omega}\psi \left(  x,u,v\right)  \eta_{2}dx\geq0
\end{array}
\right.  \text{ (resp.}\left \{
\begin{array}
[c]{c}%
\leq0\\
\leq0
\end{array}
\text{)}\right.
\]
for all nonnegative functions $\eta_{1}$, $\eta_{2}\in \mathcal{W}_{0}%
^{1,p}\left(  \Omega \right)  $.
\end{definition}

\begin{proposition}
If a pair of functions $\left(  u,v\right)  $ is a supersolution (resp.
subsolution) of (\ref{1eq1}). then the pair of functions $\left(
u+\alpha,v+\beta \right)  $(resp. $\left(  u-\alpha,v-\beta \right)  $) is also
a supersolution (resp. subsolution) of (\ref{1eq1}) for every $\alpha>0$ and
$\beta \geq0$ .
\end{proposition}

\begin{proof}
Let $\eta_{1}$, $\eta_{2}\in \mathcal{W}_{0}^{1,p}\left(  \Omega \right)  $ two
non-negative tests functions, $\alpha>0$ and $\beta \geq0$ . Then
\[
\left \{
\begin{array}
[c]{c}%
\int_{\Omega}\left \vert \nabla \left(  u+\alpha \right)  \right \vert
^{p-2}\nabla \left(  u+\alpha \right)  \nabla \eta_{1}dx+\int_{\Omega}%
\varphi \left(  x,u+\alpha,v+\beta \right)  \eta_{1}dx=\\
\int_{\Omega}\left \vert \nabla \left(  u+\alpha \right)  \right \vert
^{p-2}\nabla \left(  u+\alpha \right)  \nabla \eta_{1}dx+\int_{\Omega}%
\varphi \left(  x,u,v\right)  \eta_{1}dx+\\
\int_{\Omega}\left(  \varphi \left(  x,u+\alpha,v+\beta \right)  -\varphi \left(
x,u,v+\beta \right)  \right)  \eta_{1}dx+\int_{\Omega}\left(  \varphi \left(
x,u,v+\beta \right)  -\varphi \left(  x,u,v\right)  \right)  \eta_{1}dx
\end{array}
\right.
\]

and
\[
\left \{
\begin{array}
[c]{c}%
\int_{\Omega}\left \vert \nabla \left(  v+\beta \right)  \right \vert ^{p-2}%
\nabla \left(  v+\beta \right)  \nabla \eta_{2}dx+\int_{\Omega}\psi \left(
x,u+\alpha,v+\beta \right)  \eta_{2}dx=\\
\int_{\Omega}\left \vert \nabla \left(  v+\beta \right)  \right \vert ^{p-2}%
\nabla \left(  v+\beta \right)  \nabla \eta_{2}dx+\int_{\Omega}\psi \left(
x,u,v\right)  dx+\\
\int_{\Omega}\left(  \psi \left(  x,u+\alpha,v+\beta \right)  -\psi \left(
x,u,v+\beta \right)  \right)  \eta_{2}dx+\int_{\Omega}\left(  \psi \left(
x,u,v+\beta \right)  -\psi \left(  x,u,v\right)  \right)  \eta_{2}dx
\end{array}
\right.
\]
since $\nabla \left(  u+\alpha \right)  =\nabla u$ and $\nabla \left(
v+\beta \right)  =\nabla v$ , then
\begin{align*}
& \int_{\Omega}\left \vert \nabla \left(  u+\alpha \right)  \right \vert
^{p-2}\nabla \left(  u+\alpha \right)  \nabla \eta_{1}dx+\int_{\Omega}%
\varphi \left(  x,u+\alpha,v+\beta \right)  \eta_{1}dx=\\
& \int_{\Omega}\left \vert \nabla u\right \vert ^{p-2}\nabla u\nabla \eta
_{1}dx+\int_{\Omega}\varphi \left(  x,u,v\right)  \eta_{1}dx+\int_{\Omega
}\left(  \varphi \left(  x,u+\alpha,v+\beta \right)  -\varphi \left(
x,u,v+\beta \right)  \right)  \eta_{1}dx\\
& +\int_{\Omega}\left(  \varphi \left(  x,u,v+\beta \right)  -\varphi \left(
x,u,v\right)  \right)  \eta_{1}dx
\end{align*}

and
\[
\left \{
\begin{array}
[c]{c}%
\int_{\Omega}\left \vert \nabla \left(  v+\beta \right)  \right \vert ^{p-2}%
\nabla \left(  v+\beta \right)  \nabla \eta_{2}dx+\int_{\Omega}\psi \left(
x,u+\alpha,v+\beta \right)  \eta_{2}dx=\\
\int_{\Omega}\left \vert \nabla v\right \vert ^{p-2}\nabla v\nabla \eta
_{2}dx+\int_{\Omega}\psi \left(  x,u,v\right)  dx+\\
\int_{\Omega}\left(  \psi \left(  x,u+\alpha,v+\beta \right)  -\psi \left(
x,u,v+\beta \right)  \right)  \eta_{2}dx+\int_{\Omega}\left(  \psi \left(
x,u,v+\beta \right)  -\psi \left(  x,u,v\right)  \right)  \eta_{2}dx
\end{array}
\right.
\]

using $\left(  H_{2}\right)  $ and the fact that $\left(  u,v\right)  $ is a
supersolution, we get
\[
\left \{
\begin{array}
[c]{c}%
\int_{\Omega}\left \vert \nabla \left(  u+\alpha \right)  \right \vert
^{p-2}\nabla \left(  u+\alpha \right)  \nabla \eta_{1}dx+\int_{\Omega}%
\varphi \left(  x,u+\alpha,v+\beta \right)  \eta_{1}dx\geq0\\
\int_{\Omega}\left \vert \nabla \left(  v+\beta \right)  \right \vert ^{p-2}%
\nabla \left(  v+\beta \right)  \nabla \eta_{2}dx+\int_{\Omega}\psi \left(
x,u+\alpha,v+\beta \right)  \eta_{2}dx\geq0
\end{array}
\right.
\]

which shows that $\left(  u+\alpha,v+\beta \right)  $ is a supersolution. By
the same way we show that $\left(  u-\alpha,v-\beta \right)  $ is a subsolution.
\end{proof}

\section{Variational Dirichlet problem}

\bigskip We put:

$s:=\frac{dr\left(  p-1\right)  }{d-pr}$ where $r\in \left[  \frac
{dp^{^{\prime}}}{d+p^{^{\prime}}};\frac{d}{p}\right]  $, $1<p<d$ and $\frac
{d}{p}\leq p^{^{\prime}}$ for some $d.$

Let $f$ $,$ $g\in L^{r}\left(  \Omega \right)  \cap L^{p^{^{\prime}}}\left(
\Omega \right)  $ and $h,k\in \mathcal{W}^{1,p}\left(  \Omega \right)  $. For
$f\in L^{r}\left(  \Omega \right)  $ let $\widetilde{u}_{f}\in \mathcal{W}%
_{o}^{1,p}\left(  \Omega \right)  $ the solution of
\[
\left \{
\begin{array}
[c]{c}%
\triangle_{p}u=f\text{ \ in }\Omega \text{,}\\
u-h\text{ }\in \mathcal{W}_{o}^{1,p}\left(  \Omega \right)  \text{.}%
\end{array}
\right.
\]

and for $g\in L^{r}\left(  \Omega \right)  $ let $\  \widetilde{v}_{g}%
\in \mathcal{W}^{1,p}\left(  \Omega \right)  $ the solution of
\[
\left \{
\begin{array}
[c]{c}%
\triangle_{p}v=g\text{\ in }\Omega \text{,}\\
v-k\text{ }\in \mathcal{W}_{o}^{1,p}\left(  \Omega \right)  \text{.}%
\end{array}
\right.
\]

Consider the following Dirichlet problem%
\[
\left(  P\right)  \left \{
\begin{array}
[c]{c}%
\begin{array}
[c]{c}%
\mathcal{L}_{1}\left(  u,v\right)  :=-\triangle_{p}u+\varphi \left(
x,u,v\right)  =0\text{ a.e. }x\in \Omega \text{,}\\
\mathcal{L}_{2}\left(  u,v\right)  :=-\triangle_{p}v+\psi \left(  x,u,v\right)
=0\text{ a.e. }x\in \Omega \text{,}%
\end{array}
\\
u-h\in \mathcal{W}_{0}^{1,p}\left(  \Omega \right)  \text{ and }v-k\in
\mathcal{W}_{0}^{1,p}\left(  \Omega \right)  \text{ on }\partial \Omega \text{.}%
\end{array}
\right.
\]

We set $\widetilde{u}=u-h$ and $\widetilde{v}=v-k$, then we take back to the
following homogeneous variational problem:
\[
\left(  \widetilde{P}\right)  \left \{
\begin{array}
[c]{c}%
\begin{array}
[c]{c}%
\mathcal{L}_{1}\left(  \widetilde{u},\widetilde{v}\right)  \left(  x\right)
:=-\triangle_{p}\left(  \widetilde{u}+h\right)  +\widetilde{\varphi}\left(
x,\widetilde{u},\widetilde{v}\right)  =0\text{ a.e. }x\in \Omega \\
\mathcal{L}_{2}\left(  \widetilde{u},\widetilde{v}\right)  \left(  x\right)
:=-\triangle_{p}\left(  \widetilde{v}+k\right)  +\widetilde{\psi}\left(
x,\widetilde{u},\widetilde{v}\right)  =0\text{ a.e. }x\in \Omega
\end{array}
\\
\widetilde{u}\in \mathcal{W}_{0}^{1,p}\left(  \Omega \right)  ,\widetilde{v}%
\in \mathcal{W}_{0}^{1,p}\left(  \Omega \right)  \text{ .}%
\end{array}
\right.
\]

where%

\begin{align*}
\widetilde{\varphi}\left(  x,u,v\right)   & =\varphi \left(  x,u+h,v+k\right)
\text{,}\\
\widetilde{\psi}\left(  x,u,v\right)   & =\psi \left(  x,u+h,v+k\right)
\text{.}%
\end{align*}

\begin{remark}

\begin{itemize}
\item Using $\left(  H_{1}\right)  $ and inequality
\[
\left \vert a+b\right \vert ^{p}\leq \left \{
\begin{array}
[c]{c}%
\left(  1+\mathcal{\varepsilon}\right)  ^{p-1}\left \vert a\right \vert
^{p}+\left(  1+\frac{1}{\mathcal{\varepsilon}}\right)  ^{p-1}\left \vert
b\right \vert ^{p}\text{ for \ }1\leq p<\infty \text{;}\\
\left \vert a\right \vert ^{p}+\left \vert b\right \vert ^{p}\text{ for
\ }0<p<1\text{.}%
\end{array}
\right.
\]

for arbitrary $a,b\in$ $%
%TCIMACRO{\U{211d} }%
%BeginExpansion
\mathbb{R}
%EndExpansion
$ and $\mathcal{\varepsilon}>0$, we obtain that $\  \widetilde{\varphi
},\widetilde{\psi}$ $:\Omega \times%
%TCIMACRO{\U{211d} }%
%BeginExpansion
\mathbb{R}
%EndExpansion
\times%
%TCIMACRO{\U{211d} }%
%BeginExpansion
\mathbb{R}
%EndExpansion
\rightarrow%
%TCIMACRO{\U{211d} }%
%BeginExpansion
\mathbb{R}
%EndExpansion
$ are given Carath\'{e}odory functions satisfying: $\ $%
\[
\left(  \widetilde{H}\right)  \left \{
\begin{array}
[c]{c}%
\left \vert \widetilde{\varphi}\left(  x,u,v\right)  \right \vert \leq
a_{1}^{^{\prime}}\left \vert u\right \vert ^{p-1}+a_{2}^{^{\prime}}\left \vert
v\right \vert ^{p-1}+c\left(  x\right)  \text{;}\\
\left \vert \widetilde{\psi}\left(  x,u,v\right)  \right \vert \leq
b_{1}^{^{\prime}}\left \vert u\right \vert ^{p-1}+b_{2}^{^{\prime}}\left \vert
v\right \vert ^{p-1}+c^{^{\prime}}\left(  x\right)  \text{.}%
\end{array}
\right.
\]

\end{itemize}

almost everywhere $x\in \Omega$. Where
\begin{align*}
a_{1}^{^{\prime}}  & =a_{1}\left(  1+\mathcal{\varepsilon}\right)
^{p-1}\text{, }a_{2}^{^{\prime}}=a_{2}\left(  1+\mathcal{\varepsilon}\right)
^{p-1}\text{ , }\\
b_{1}^{^{\prime}}  & =b_{1}\left(  1+\mathcal{\varepsilon}\right)
^{p-1}\text{, }b_{2}^{^{\prime}}=b_{2}\left(  1+\mathcal{\varepsilon}\right)
^{p-1}\text{,}\\
c\left(  x\right)   & =a_{1}\left(  1+\frac{1}{\mathcal{\varepsilon}}\right)
^{p-1}\left \vert h\left(  x\right)  \right \vert ^{p-1}+a_{2}\left(  1+\frac
{1}{\mathcal{\varepsilon}}\right)  ^{p-1}\left \vert k\left(  x\right)
\right \vert ^{p-1}\text{,}\\
c^{^{\prime}}\left(  x\right)   & =b_{1}\left(  1+\frac{1}%
{\mathcal{\varepsilon}}\right)  ^{p-1}\left \vert h\left(  x\right)
\right \vert ^{p-1}+b_{2}\left(  1+\frac{1}{\mathcal{\varepsilon}}\right)
^{p-1}\left \vert k\left(  x\right)  \right \vert ^{p-1}\text{.}%
\end{align*}

\begin{itemize}
\item If $\left(  \widetilde{u},\widetilde{v}\right)  $ is a solution of
$\left(  \widetilde{P}\right)  $. $\left(  \widetilde{u}+h,\widetilde
{v}+k\right)  $ is a solution of $(P)$.

\item $\widetilde{\varphi}\left(  x,u,v\right)  $, $\widetilde{\psi}\left(
x,u,v\right)  \in L^{p^{^{\prime}}}\left(  \Omega \right)  $.
\end{itemize}
\end{remark}

\begin{theorem}
If $\  \left \vert \Omega \right \vert $ is small enough. Then system $(P)$ admits
a weak solution.
\end{theorem}

\textbf{Proof\bigskip}

We consider the operator:%
\[%
\begin{array}
[c]{ccc}%
T: & L^{r}\left(  \Omega \right)  \times L^{r}\left(  \Omega \right)
\rightarrow & L^{p}\left(  \Omega \right)  \times L^{p}\left(  \Omega \right) \\
& \text{ \  \  \  \  \  \  \  \  \  \  \ }\left(  f,g\right)  \mapsto & \left(
\widetilde{u}_{f},\widetilde{v}_{g}\right)
\end{array}
\]

\begin{lemma}
The operator $%
\begin{array}
[c]{ccc}%
T: & L^{r}\left(  \Omega \right)  \times L^{r}\left(  \Omega \right)
\rightarrow & L^{p}\left(  \Omega \right)  \times L^{p}\left(  \Omega \right)
\end{array}
$ is completely continuous.
\end{lemma}

\begin{proof}
By the regularity theory we have $\widetilde{u}_{f}$ and $\widetilde{v}_{g}$
are continuous functions, so $%
\begin{array}
[c]{ccc}%
T: & L^{r}\left(  \Omega \right)  \times L^{r}\left(  \Omega \right)
\rightarrow & \mathcal{W}_{0}^{1,p}\left(  \Omega \right)  \times
\mathcal{W}_{0}^{1,p}\left(  \Omega \right)
\end{array}
$ is a continuous operator, on the other hand $\mathcal{W}_{0}^{1,p}\left(
\Omega \right)  \hookrightarrow L^{p}\left(  \Omega \right)  $ is compact, hence
the result.
\end{proof}

Let $\mathcal{B}_{1}\left(  u,v\right)  \left(  .\right)  :=\widetilde
{\varphi}\left(  .,u\left(  .\right)  ,v\left(  .\right)  \right)  $ the
Nemytskii operator associated to $\widetilde{\varphi}$, and $\mathcal{B}%
_{2}\left(  u,v\right)  \left(  .\right)  :=\widetilde{\psi}\left(  .,u\left(
.\right)  ,v\left(  .\right)  \right)  $ the Nemytskii operator associated to
$\widetilde{\psi}$.

\begin{lemma}
The operators $%
\begin{array}
[c]{ccc}%
\mathcal{B}_{1}: & L^{p}\left(  \Omega \right)  \times L^{p}\left(
\Omega \right)  \rightarrow & L^{r}\left(  \Omega \right)
\end{array}
$, and $%
\begin{array}
[c]{ccc}%
\mathcal{B}_{2}: & L^{p}\left(  \Omega \right)  \times L^{p}\left(
\Omega \right)  \rightarrow & L^{r}\left(  \Omega \right)
\end{array}
$ are continuous and bounded.
\end{lemma}

\begin{proof}
We see that $\frac{r}{p\left(  p-1\right)  }+\frac{p\left(  p-1\right)
-r}{p\left(  p-1\right)  }=1$, using Young inequality in $\left(
\widetilde{H}\right)  $ with these exponents, we get
\[
\text{\ }\left(  \widetilde{H}_{1}\right)  \text{ \  \ }\left \vert
\widetilde{\varphi}\left(  x,u,v\right)  \right \vert \text{ }\leq \frac
{r}{p\left(  p-1\right)  }\left \vert u\right \vert ^{\frac{p}{r}}+\frac
{r}{p\left(  p-1\right)  }\left \vert v\right \vert ^{\frac{p}{r}}+a\left(
x\right)  \text{ \  \  \  \ }%
\]

where $a\left(  x\right)  =\frac{p\left(  p-1\right)  -r}{p\left(  p-1\right)
}a_{1}^{^{\prime^{\frac{p\left(  p-1\right)  }{p\left(  p-1\right)  -r}}}%
}+\frac{p\left(  p-1\right)  -r}{p\left(  p-1\right)  }a_{2}^{^{\prime
^{\frac{p\left(  p-1\right)  }{p\left(  p-1\right)  -r}}}}+c\left(  x\right)
.$ Since $r\leq p^{^{\prime}}$ and $\left \vert \Omega \right \vert <\infty$ then
$L^{p^{^{\prime}}}\hookrightarrow L^{r}$ , so $\left \vert h\right \vert
^{p-1},\left \vert k\right \vert ^{p-1}\in L^{r}$ consequently $a\in L^{r}.$
According to \cite[proposition 26.6]{Zei90} an increase of type $\left(
\widetilde{H}_{1}\right)  $ is sufficient for the operator $\mathcal{B}_{1}$
to be continuous and bounded from $L^{p}\left(  \Omega \right)  \times
L^{p}\left(  \Omega \right)  $ to $L^{r}\left(  \Omega \right)  $, by the same
way we show that $\mathcal{B}_{2}$ is continuous and bounded from
$L^{p}\left(  \Omega \right)  \times L^{p}\left(  \Omega \right)  $ to
$L^{r}\left(  \Omega \right)  .$
\end{proof}

\begin{proof}
[Proof of the thoerem]Let $\Lambda \left(  f,g\right)  :=\left(  \mathcal{B}%
_{1}\circ T\left(  f,g\right)  ;\mathcal{B}_{2}\circ T\left(  f,g\right)
\right)  $

For $M>0$, we put:%
\[
\mathcal{K}_{M}=\left \{  \left(  f,g\right)  \in L^{r}\left(  \Omega \right)
\times L^{r}\left(  \Omega \right)  :\left \Vert \left(  f,g\right)  \right \Vert
_{r\times r}\leq M\right \}
\]

where $\left \Vert \left(  f,g\right)  \right \Vert _{r\times r}=\max \left \{
\left \Vert f\right \Vert _{r}\text{ };\left \Vert g\right \Vert _{r}\right \}  .$

For $\left(  f,g\right)  \in \mathcal{K}_{M}$, we have:
\begin{align*}
\left \Vert \Lambda \left(  f,g\right)  \right \Vert _{r\times r}  & =\left \Vert
\left(  \mathcal{B}_{1}\circ T\left(  f,g\right)  ;\mathcal{B}_{2}\circ
T\left(  f,g\right)  \right)  \right \Vert _{r\times r}\\
& =\max \left \{  \left \Vert \mathcal{B}_{1}\left(  \widetilde{u}_{f}%
,\widetilde{v}_{g}\right)  \right \Vert _{r}\text{ };\left \Vert \mathcal{B}%
_{2}\left(  \widetilde{u}_{f},\widetilde{v}_{g}\right)  \right \Vert
_{r}\right \}  .
\end{align*}

Using $\left(  \widetilde{H}\right)  $ we get:
\[
\left \Vert \mathcal{B}_{1}\left(  \widetilde{u}_{f},\widetilde{v}_{g}\right)
\right \Vert _{r}\leq a_{1}^{^{\prime}}\left \Vert \widetilde{u}_{f}\right \Vert
_{r\left(  p-1\right)  }^{p-1}+a_{2}^{^{\prime}}\left \Vert \widetilde{v}%
_{g}\right \Vert _{r\left(  p-1\right)  }^{p-1}+\left \Vert c\right \Vert _{r}%
\]

according to (\cite[theorem 2.5]{DaDr09} ) we know that there exist $C>0$ such
that
\[
\left \Vert \widetilde{u}_{f}\right \Vert _{s}^{p-1}\leq C\left \Vert
f\right \Vert _{r}%
\]

and
\[
\left \Vert \widetilde{v}_{g}\right \Vert _{s}^{p-1}\leq C\left \Vert
g\right \Vert _{r}%
\]

Since $s>r\left(  p-1\right)  $ and $\left \vert \Omega \right \vert <\infty$,
the H\"{o}lder inequality implique that
\begin{align*}
\left \Vert \widetilde{u}_{f}\right \Vert _{r\left(  p-1\right)  }^{p-1}  &
\leq \left \vert \Omega \right \vert ^{\frac{1}{r}-\frac{\left(  p-1\right)  }{s}%
}\left \Vert \widetilde{u}_{f}\right \Vert _{s}^{p-1}\\
& \leq C\left \vert \Omega \right \vert ^{\frac{p}{d}}\left \Vert f\right \Vert
_{r}%
\end{align*}

and%

\begin{align*}
\left \Vert \widetilde{v}_{g}\right \Vert _{r\left(  p-1\right)  }^{p-1}  &
\leq \left \vert \Omega \right \vert ^{\frac{1}{r}-\frac{\left(  p-1\right)  }{s}%
}\left \Vert \widetilde{v}_{g}\right \Vert _{s}^{p-1}\\
& \leq C\left \vert \Omega \right \vert ^{\frac{p}{d}}\left \Vert g\right \Vert
_{r}%
\end{align*}

so
\[
\left \Vert \mathcal{B}_{1}\circ \left(  f,g\right)  \right \Vert _{r}\leq
C\left \vert \Omega \right \vert ^{\frac{p}{d}}\left(  a_{1}^{^{\prime}%
}\left \Vert f\right \Vert _{r}+a_{2}^{^{\prime}}\left \Vert g\right \Vert
_{r}\right)  +\left \Vert c\right \Vert _{r}%
\]

and by the same way, we get
\[
\left \Vert \mathcal{B}_{2}\circ \left(  f,g\right)  \right \Vert _{r}\leq
C\left \vert \Omega \right \vert ^{\frac{p}{d}}\left(  b_{1}^{^{\prime}%
}\left \Vert f\right \Vert _{r}+b_{2}^{^{\prime}}\left \Vert g\right \Vert
_{r}\right)  +\left \Vert c^{^{\prime}}\right \Vert _{r}%
\]

since $\left \vert \Omega \right \vert $ is small enough we can choise
$\left \vert \Omega \right \vert $ such that:
\[
\lambda:=\max \left \{  a_{1}^{^{\prime}},a_{2}^{^{\prime}},b_{1}^{^{\prime}%
},b_{2}^{^{\prime}}\right \}  C\left \vert \Omega \right \vert ^{\frac{p}{d}}<1
\]
hence
\[
\left \Vert \mathcal{B}_{1}\circ \left(  f,g\right)  \right \Vert _{r}\leq
\lambda \max \left \{  \left \Vert f\right \Vert _{r},\left \Vert g\right \Vert
_{r}\right \}  +\left \Vert c\right \Vert _{r}%
\]

and
\[
\left \Vert \mathcal{B}_{2}\circ \left(  f,g\right)  \right \Vert _{r}\leq
\lambda \max \left \{  \left \Vert f\right \Vert _{r},\left \Vert g\right \Vert
_{r}\right \}  +\left \Vert c^{^{\prime}}\right \Vert _{r}%
\]

so
\[
\max \left \{  \left \Vert \mathcal{B}_{1}\circ \left(  f,g\right)  \right \Vert
_{r}\text{ };\left \Vert \mathcal{B}_{2}\circ \left(  f,g\right)  \right \Vert
_{r}\right \}  \leq \lambda \max \left \{  \left \Vert f\right \Vert _{r},\left \Vert
g\right \Vert _{r}\right \}  +\max \left \{  \left \Vert c\right \Vert
_{r};\left \Vert c^{^{\prime}}\right \Vert _{r}\right \}
\]

then
\begin{align*}
\left \Vert \Lambda \left(  f,g\right)  \right \Vert _{r\times r}  & \leq
\lambda \max \left \{  \left \Vert f\right \Vert _{r},\left \Vert g\right \Vert
_{r}\right \}  +\max \left \{  \left \Vert c\right \Vert _{r};\left \Vert
c^{^{\prime}}\right \Vert _{r}\right \} \\
& \leq \lambda M+\max \left \{  \left \Vert c\right \Vert _{r};\left \Vert
c^{^{\prime}}\right \Vert _{r}\right \}
\end{align*}

putting
\[
M_{0}:=\frac{\max \left \{  \left \Vert c\right \Vert _{r};\left \Vert c^{^{\prime
}}\right \Vert _{r}\right \}  }{1-\lambda}%
\]

hence $\Lambda \left(  \mathcal{K}_{M}\right)  \subset \mathcal{K}_{M}$,
$\forall M\geq M_{0}$.

$K_{M}$ is a no empty closed convex subset of $\ L^{r}\left(  \Omega \right)
\times L^{r}\left(  \Omega \right)  $ and $%
\begin{array}
[c]{ccc}%
\Lambda: & L^{r}\left(  \Omega \right)  \times L^{r}\left(  \Omega \right)
\rightarrow & L^{r}\left(  \Omega \right)  \times L^{r}\left(  \Omega \right)
\end{array}
$ is competly continuos, so by Shauder fixed point theorem $\Lambda$ admits a
fixed point in $\mathcal{K}_{M}$, hence the problem $\left(  P\right)  $ admet
a weak solution.
\end{proof}

By the regularity theory \cite[Corollary 4.10]{MaZi97}, any bounded solution
of (\ref{1eq1}) can be redefined in a set of measure zero so that it becomes continuous.

\bigskip

\bigskip

\bigskip

\bigskip

\bigskip

\bigskip

\bigskip

\bigskip

\bigskip

\bigskip

\bigskip

\bigskip

\bigskip
\end{document}